\documentclass[a4paper, 10pt]{amsart}
\usepackage{amsmath, amssymb, color, hyperref}

\theoremstyle{plain}
\newtheorem{theorem}{\bf Theorem}[section]
\newtheorem{proposition}[theorem]{\bf Proposition}
\newtheorem{lemma}[theorem]{\bf Lemma}
\newtheorem{corollary}[theorem]{\bf Corollary}
\newtheorem{conjecture}[theorem]{\bf Conjecture}

\theoremstyle{definition}

\newtheorem{remark}[theorem]{\bf Remark}

\def\Dec{\mathrm{Dec}}
\DeclareMathOperator{\gp}{gp}

 \DeclareMathOperator{\Int}{Int}
\DeclareMathOperator{\spec}{spec}

\DeclareMathOperator{\rev}{Rev}

\newcommand{\N}{\mathbb N}
\newcommand{\Z}{\mathbb Z}
\newcommand{\R}{\mathbb R}
\newcommand{\Q}{\mathbb Q}

\newcommand{\LK}{\,[\![}
\newcommand{\RK}{]\!]}

\newcommand{\Pp}{{\mathbb P}}

\newcommand{\red}{{\text{\rm red}}}
\renewcommand{\t}{\, | \,}
\newcommand{\fin}{\text{\rm fin}}

\newcommand{\Pfz}{\mathcal P_{\fin,0} (\N_0)}
\newcommand{\Pfze}{\mathcal P_{\fin,0} }
\newcommand{\Pf}{\mathcal P_{\fin}}
\newcommand{\Pfn}{\mathcal P_{\fin} (\N_0)}
\newcommand{\abs}[1]{\lvert #1 \rvert}

\numberwithin{equation}{section}

\begin{document}

\title[Power Monoids of numerical monoids]{On algebraic properties of  \\ power monoids of numerical monoids}

\author{Pierre-Yves Bienvenu and Alfred Geroldinger}

\address{Graz University of Technology, NAWI Graz \\
Institute of Analysis and Number Theory \\
Kopernikusgasse 24/2\\
8010 Graz, Austria}
\email{bienvenu@math.tugraz.at}
\urladdr{https://www.math.tugraz.at/~bienvenu/}

\address{University of Graz, NAWI Graz \\
Institute for Mathematics and Scientific Computing \\
Heinrichstra{\ss}e 36\\
8010 Graz, Austria}
\email{alfred.geroldinger@uni-graz.at}
\urladdr{https://imsc.uni-graz.at/geroldinger}

\thanks{This work was supported by the Austrian-French Project ArithRand (FWF  I 4945-N and  ANR-20-CE91-0006) and by the Austrian Science Fund FWF, Project Number P33499.}


\begin{abstract}
Let $S \subset \mathbb{N}_0$ be a numerical monoid and let $\mathcal P_{\mathrm{fin}} (S)$, resp. $\mathcal P_{\mathrm{fin},0}(S)$, denote the power monoid, resp. the restricted power monoid, of $S$, that is the set of all finite nonempty subsets of $S$, resp. the set of all finite nonempty subsets of $S$ containing 0,  with set addition as operation. The arithmetic of power monoids received  some attention  in recent literature. We complement these investigations by studying algebraic properties  of power monoids, such as their prime spectrum. Moreover, we show that almost all elements of $\mathcal P_{\mathrm{fin},0} (S)$ are irreducible (i.e., they are not proper sumsets).
\end{abstract}


\maketitle


\section{Introduction} \label{1}

Set addition,  with its innumerable facets, is a central topic in arithmetic combinatorics. Factorization theory studies, for a given monoid or domain, factorizations of elements into irreducible (indecomposable) elements. Both areas are closely connected, for example, by the factorization theory of Krull monoids (\cite{Ge-Ru09, Sc16a}). Tringali pushed forward a new connection by initiating investigations of the arithmetic of power monoids (\cite{Fa-Tr18a, An-Tr21a, Tr22a}), which are defined as follows. For an additive abelian monoid $S$, say for a  numerical monoid,  let $\mathcal P_{\fin} (S)$  denote the set of all finite nonempty subsets of $S$. Together with set addition as operation, $\mathcal P_{\fin} (S)$ is easily seen to be abelian monoid, called the power monoid of $S$, and $\{0_S\}$ is its zero-element.
While, clearly, being of interest in its own rights, the arithmetic of power monoids is connected with the arithmetic of other monoids, such as the monoid of ideals of polynomial rings (\cite[Proposition 5.13]{Ge-Kh22b}).

To fix further terminology, let now $D$ be an additively written monoid (resp. a multiplicatively written monoid, for example a domain). If an element $a \in D$ can be written as a sum $a = u_1 + \ldots + u_k $ (respectively, as a  product $a = u_1 \cdot \ldots \cdot u_k$) of $k$ irreducible elements, then $k$ is called a factorization length of $a$ and the set $\mathsf L (a) \subset \N_0$ denotes the set of all factorization lengths of $a$.
The system $\mathcal L (D) = \{\mathsf L (a) \colon a \in D\}$ of all sets of lengths is a key arithmetic invariant. Algebraic finiteness conditions on $D$ (such as the finiteness of the class group in case when $D$ is a Krull domain) guarantee that sets of lengths are highly structured (see \cite[Chapter 4]{Ge-HK06a} for an overview). On the other hand, there are various classes of monoids with the property that every finite nonempty subset of $\N_{\ge 2}$ occurs as a set of lengths. These classes include Krull monoids with infinite class group having prime divisors in all classes, rings of integer valued polynomials $\Int (D)$, where $D$ is a Dedekind domain having infinitely many maximal ideals of finite index, and others (see \cite{Ka99a, Fr13a, Fr-Na-Ri19a, Go19a, Ch-Fa-Wi22a, Fa-Zh22a, Fa-Fr-Wi23} and \cite{Ge-Zh20a} for a survey).

According to a conjecture of Fan and Tringali, power monoids of large classes of monoids have the same property. We formulate a simple version of this conjecture (see \cite[Section 5]{Fa-Tr18a}).

\smallskip
\begin{conjecture}[{\bf Fan \& Tringali}] \label{1.1}
For every numerical monoid $S$ and every finite nonempty subset $L \subset \N_{\ge 2}$,  there is a finite nonempty set $A \subset S$   such that for its set of lengths $\mathsf L (A)$, with respect to the power monoid $\mathcal P_{\fin} (S)$, we have $\mathsf L (A) = L$.
\end{conjecture}

\smallskip
It is comparatively easy to show that, if the above conjecture holds true for the power monoid of numerical monoids, then the analogue conjecture holds true for power monoids of further large classes of monoids (see \cite{Fa-Tr18a}). The conjecture is backed up by a series of results, some of which we gather in the following theorem.  For simplicity, we formulate the results for the power monoid of $\N_0$ (see Section \ref{2} for the involved definitions and  \cite[Theorem 4.11]{Fa-Tr18a} and \cite[Proposition 5.3]{Ge-Kh22b} for proofs). All results of Theorem \ref{1.2} are simple consequences of Conjecture \ref{1.1}, if it holds true.

\smallskip
\begin{theorem} \label{1.2}
The monoid $\mathcal P_{\fin} (\N_0)$ of all finite nonempty subsets of the nonnegative integers, with set addition as operation, has the following properties.
\begin{enumerate}
\item For its set of distances,  we have $\Delta \big( \mathcal P_{\fin} (\N_0) \big) = \N$.

\item For the unions of sets of lengths, we have $\mathcal U_k \big( \mathcal P_{\fin} (\N_0) \big) = \N_{\ge 2}$ for all $k \ge 2$.

\item For every rational number $q \ge 1$, there is a finite nonempty set $A \subset \N_0$ such that for its set of lengths $\mathsf L (A)$ we have
      \[
      q = \max \mathsf L (A)/\min \mathsf L (A) \,.
      \]
\end{enumerate}
\end{theorem}

\smallskip
In the present paper, we pursue a new strategy. Instead of studying further arithmetic invariants to back up Conjecture \ref{1.1} and to proceed by combinatorial ad-hoc constructions, we  investigate algebraic properties of power monoids $\mathcal P_{\fin} (S)$ of numerical monoids $S$, with a focus on the nonnegative integers. Along our way, we compare the algebraic properties of power monoids with algebraic properties of Krull monoids and of rings of integer-valued polynomials. These two classes are among the best understood objects in factorization theory. The study of their arithmetic is based on a solid understanding of their algebraic properties. Thus, apart from being of interest in its own right, a better understanding of algebraic properties of power monoids should pave the way for a better understanding of their arithmetic.

In Section \ref{2}, we gather the required background on monoids. Section \ref{3} deals with basic structural properties of power monoids. In Section \ref{4}, we study the prime spectrum of power monoids (main results are Theorems \ref{4.3}, \ref{4.4}, and \ref{4.8}). In Section \ref{5}, we show that some finer arithmetic invariants (the $\omega$-invariants) are infinite, which supports Conjecture \ref{1.1} (see Corollary \ref{5.2} and Remark \ref{5.3}). In Section \ref{6}, we prove that almost all elements in the restricted power monoid of a numerical monoid are atoms (Theorems \ref{th:density} and \ref{6.1}), a property which is in strong contrast to the density of atoms in all monoids studied so far.

\smallskip
\section{Background on monoids} \label{2}
\smallskip

We denote by $\N$ the set of positive integers and we set $\N_0 = \N \cup \{0\}$. For real numbers $a, b \in \R$, we let $[a, b ] = \{x \in \Z \colon a \le x \le b \}$ denote the discrete interval between $a$ and $b$. Let $A$ and  $B$ be sets. We use the symbol $A \subset B$ to mean that $A$ is contained in $B$ but may be equal to $B$.  Suppose that $A$ and $B$ are  subsets of $\Z$. Then $A+B = \{a+b \colon a \in A, b \in B \}$ denotes their sumset and $A-B = \{a-b \colon a \in A, b \in B\}$ denotes their difference set. The set of distances $\Delta (A) \subset \N$ is the set of all $d \in \N$ for which there is $a \in A$ such that $A \cap [a, a+d] = \{a, a+d\}$. For every $k \in \N$, $kA = A + \ldots + A$ is the $k$-fold sumset of $A$ and $k \cdot A = \{ka \colon a \in A\}$ is the dilation of $A$ by $k$. For $k=0$, we set $kA = k \cdot A = \{0\}$.

Let $H$ be a  commutative semigroup with identity element. In this manuscript, we will consider both additively written semigroups (such as the power monoid of numerical monoids) and multiplicatively written semigroups (such as the multiplicative semigroup of nonzero elements of a domain). In this introductory section, we use multiplicative notation since it is more common in factorization theory. We denote by $H^{\times}$ the group of invertible elements of $H$. We say that $H$ is reduced if $H^{\times} = \{1\}$ and we denote by $H_\red=H/H^{\times} = \{ a H^{\times} \colon a \in H\}$ the associated reduced monoid of $H$. An element $a \in H$ is said to be
\begin{itemize}
\item {\it cancellative} if $b, c \in H$ and $ab = ac$ implies that $b=c$, and

\item {\it unit-cancellative} if $a \in H$ and $a = au$ implies that $u \in H^{\times}$.
\end{itemize}
By definition, every cancellative element is unit-cancellative. The semigroup $H$ is said to be {\it cancellative} (resp. {\it unit-cancellative}) if every element of $H$ is cancellative (resp. unit-cancellative).

\medskip
\centerline{\it Throughout this manuscript, a monoid means a}
\centerline{\it commutative unit-cancellative semigroup with identity element.}
\medskip

For a set $P$, we denote by $\mathcal F (P)$ the free abelian monoid with basis $P$. Elements $a\in\mathcal F(P)$ are written in the form
\[
a=\prod_{p\in P} p^{\mathsf v_p(a)}\,,\quad\text{where $\mathsf v_p\colon\mathcal F(P)\to\N_0$ }
\]
is the $p$-adic valuation, and we denote by $|a|=\sum_{p\in P}\mathsf v_p(a)\in\N_0$ the length of $a$.
Let $H$ be a monoid and $S \subset H$ be a subset. If $SH = S$, then $S$ is called an {\it $s$-ideal} of $H$ and we denote by $s$-$\spec (H)$ the set of prime $s$-ideals of $H$. Note that the empty set is a prime $s$-ideal.
The set  $S$ is called {\it divisor-closed} if $a \in S$ and $b \in H$ with $b \mid a $ implies that $b \in S$. Thus, $S \subset H$ is a divisor-closed submonoid if and only if $H \setminus S$ is a prime $s$-ideal of $H$.
We denote by $\LK S \RK$ the smallest divisor-closed submonoid containing $S$. If $S = \{a\}$ for some  $a \in H$, then
\[
\LK \{a \} \RK = \LK a \RK = \{ b \in H \colon b \ \text{divides some power of $a$} \} \subset H
\]
the smallest divisor-closed submonoid of $H$ containing $a$. The monoid $H$ is said to be
\begin{itemize}
\item {\it locally finitely generated} if $\LK a \RK_{\red} \subset H_{\red}$ is finitely generated for all $a \in H$,

\item {\it torsion-free} if $a^n = b^n$, where $a, b \in H$ and $n \in \N$, implies that $a=b$, and

\item a {\it Krull monoid} if it is cancellative, completely integrally closed, and satisfies the ascending chain condition on divisorial ideals (see \cite{HK98, Ge-HK06a} for details).
\end{itemize}
There are an abelian group  $\gp (H)$ (called the {\it Grothendieck group}) and  a monoid homomorphism $\iota \colon H \to \gp (H)$   which have the following universal property:
\begin{itemize}
\item[] For every monoid homomorphism $\varphi \colon H \to G$, where $G$ is an abelian group, there is a  group homomorphism $\psi \colon \gp (H) \to G$ such that $\varphi = \psi \circ \iota$.
\end{itemize}
If $H$ is cancellative, then $\gp (H)$ is the quotient group of $H$.

\smallskip
\noindent
{\bf Arithmetic of Monoids.}  Let $H$ be a monoid. An element $p \in H$ is said to be
\begin{itemize}
\item {\it irreducible} (an {\it atom}) if $p \notin H^{\times}$ and $p = ab$ with $a, b \in H$ implies that $a \in H^{\times}$ or $b \in H^{\times}$, and

\item {\it prime} if $p \notin H^{\times}$ and $p \mid ab$ with $a, b \in H$ implies that $p \mid a$ or $p \mid b$.
\end{itemize}
If $p \in H$ is a cancellative prime element, then
\begin{equation} \label{structure}
H = \mathcal F ( \{p\}) \times T \,, \quad \text{where} \quad T = \{ a \in H \colon p \nmid a \} \,.
\end{equation}
We denote by $\mathcal A (H)$ the set of atoms of $H$, and note that  prime elements are  irreducible.
The free abelian monoid $\mathsf Z (H) = \mathcal F ( \mathcal A (H_{\red}))$ is the {\it factorization monoid} of $H$ and  $\pi \colon \mathsf Z (H) \to H_{\red}$, defined by $\pi (u) = u$ for all $u \in \mathcal A (H_{\red})$, denotes the {\it factorization homomorphism} of $H$. For $a \in H$,
\begin{itemize}
\item $\mathsf Z_H (a) = \mathsf Z (a) = \pi^{-1} (aH^{\times}) \subset \mathsf Z (H)$ is the {\it set of factorizations} of $a$,

\item $\mathsf L_H (a) = \mathsf L (a) = \{ |z| \colon z \in \mathsf Z (a) \} \subset \N_0$ is the {\it set of lengths} of $a$, and

\item $\mathcal L (H) = \{ \mathsf L (a) \colon a \in H \}$ is the {\it system of sets of lengths} of $H$.
\end{itemize}
An  element $p \in H$ is said to be {\it absolutely irreducible} (a {\it strong atom}) if $p$ is irreducible and $|\mathsf Z (p^n)|=1$ for all $n \in \N$. Cancellative prime elements are absolutely irreducible.
We denote by
\[
\Delta (H) = \bigcup_{L \in \mathcal L (H)} \Delta (L) \ \subset \N \quad \text{the {\it set of distances} of $H$}
\]
and, for every $k \in \N$,
\[
\mathcal U_k (H) = \bigcup_{k \in L, L \in \mathcal L (H)} \ L \ \subset \N \quad \text{is the  {\it union of sets of lengths} containing $k$} \,.
\]
The monoid $H$ is said to be {\it atomic} if $\mathsf Z (a) \ne \emptyset$ for all $a \in H$.
A monoid homomorphism $\theta \colon H \to B$, where $B$ is a monoid, is called a {\it transfer homomorphism} if it satisfies the following two properties.
\begin{enumerate}
\item[{\bf (T1)}] $B = \theta(H) B^\times$  and  $\theta^{-1} (B^\times) = H^\times$.

\item[{\bf (T2)}] If $u \in H$, \ $b,\,c \in B$  and  $\theta (u) = bc$, then there exist \ $v,\,w \in H$ \ such that \ $u = vw$,  $\theta (v) \in bB^{\times}$, and  $\theta (w) \in c B^{\times}$.
\end{enumerate}
Transfer homomorphisms allow to pull back arithmetic properties from the (simpler) monoid $B$ to the monoid $H$ (the original object of interest). In particular, they preserve sets of lengths. Thus, if $\theta \colon H \to B$ is a transfer homomorphism and $a \in H$, then $\mathsf L_H (a) = \mathsf L_B \big( \theta (a) \big)$. In particular, this implies that $\theta \big ( \mathcal A (H) \big) = \mathcal A (B)$, that $\theta^{-1} \big( \mathcal A (B) \big) = \mathcal A (H)$, and that $\mathcal L (H) = \mathcal L (B)$. A monoid is said to be {\it transfer Krull} if it allows a transfer homomorphism to a Krull monoid. For more on transfer homomorphisms see \cite{Ge-Zh20a, Ba-Re22a}.

By a domain, we mean a commutative integral domain. Let $D$ be a domain. Then the multiplicative semigroup $D^{\bullet} = D \setminus \{0\}$ is a cancellative monoid. All arithmetic properties of $D^{\bullet}$ will be attributed to $D$ and, as usual,  we set $\mathcal L (D) = \mathcal L (D^{\bullet})$, and so on. If $K$ is the quotient field of $D$, then
\[
\Int (D) = \{ f \in K[X] \colon f (D) \subset D \} \subset K[X]
\]
is the ring of integer-valued polynomials over $D$. The domain $D$ is a Krull domain if and only if $D^{\bullet}$ is a Krull monoid.

\smallskip
\noindent
{\bf Submonoids of $\Z$.} Every additive submonoid of the integers is either a group, or a submonoid of $\N_0$, or a submonoid of $-\N_0$. Let $S$ be a submonoid of $\N_0$. Then $S \cong \gcd (S) \cdot S'$, where $S'$ is a submonoid of $\N_0$ with $\gcd (S')=1$. Submonoids of $\N_0$ whose greatest common divisor is equal to $1$ are called {\it numerical monoids}.  If $S$ is a numerical monoid, then $\N_0 \setminus S$ is finite, $\mathsf F (S) = \max ( \N_0 \setminus S )$ is called the {\it Frobenius number} of $S$ (with the convention that $\mathsf F (\N_0)=0$), $S$ is finitely generated, and its set of atoms $\mathcal A (S)$ is the unique minimal generating set. For a nonempty set $A \subset \N_0$, we denote by $\langle A\rangle=\bigcup_{n=0}^\infty nA$ the submonoid generated by $A$. Then, $\langle A \rangle$ is a numerical monoid if and only if $\langle A \rangle \subset \N_0$ is cofinite if and only if $\gcd(A)=1$.
Otherwise,  $\langle A \rangle $ is cofinite in $\gcd(A)\cdot\N_0$.

\smallskip
\noindent
{\bf Power Monoids.}  For an additive monoid $S \subset \Z$, we denote by
\begin{itemize}
\item $\mathcal P_{\fin} (S)$ the {\it power monoid} of $S$, that is the semigroup of all finite nonempty subsets of $S$ with set addition as operation, and by

\item  $\mathcal P_{\fin, 0} (S)$ the {\it restricted power monoid} of $S$, that is the subsemigroup of $\mathcal P_{\fin} (S)$ consisting of all finite nonempty subsets of $S$ that contain $0$.
\end{itemize}
The study of the arithmetic of power monoids (of various classes of  semigroups) was initiated by Fan and Tringali \cite{Fa-Tr18a} and continued, among others, in \cite{An-Tr21a, Tr22a}.
In the present paper, we study power monoids of numerical monoids with a focus  on power monoids of the nonnegative integers. Let $S \subset \N_0$ be a numerical monoid. Then,
both $\mathcal P_{\fin} (S)$ and $\mathcal P_{\fin, 0} (S)$, are commutative reduced unit-cancellative semigroups (whence  monoids in the present sense) and $\{0\}$ is their zero-element. Clearly, every finite nonempty subset $A \subset S$ is the sum of irreducible sets,  the number $|\mathsf Z (A)|$ of factorizations of $A$ is finite, and the set of lengths $\mathsf L (A)$ is finite.

The arithmetic of numerical monoids has received wide attention in the literature (we refer to the monograph \cite{As-GS16}, to the survey \cite{GS16a}, and to the software package GAP \cite{numericalsgps}). Since numerical monoids are finitely generated, all invariants (sets of distances, unions of sets of lengths, and all the invariants to be discussed in Section \ref{5}) are finite. Furthermore,  also some precise results are known in terms of their set of atoms. Let $S$ be a numerical monoid with $\mathcal A (S) = \{n_1, \ldots, n_t\}$ where $t \ge 2$ and $1 < n_1 < \ldots < n_t$. Then
\[
\max \big\{ \max \mathsf L (a)/\min \mathsf L (a) \colon a \in S \big\} = n_t/n_1 \quad \text{and} \quad \min \Delta (S) = \gcd (n_2-n_1, \ldots, n_t - n_{t-1} ) \,.
\]
Moreover, sets of lengths are highly structured. Indeed, there is $M \in \N_0$ such that, for every $a \in S$,
\[
\mathsf L (a) \cap [\min \mathsf L (a) + M, \max \mathsf L (a) - M]
\]
is an arithmetic progression with difference $\min \Delta (S)$. The arithmetic of the power monoid of $S$ is very different. Conjecture \ref{1.1} states that every finite nonempty subset of $\N_{\ge 2}$ occurs as a set of lengths in the power monoid of any numerical monoid. Furthermore, all arithmetic invariants studied so far have turned out to be infinite (see Theorem \ref{1.2} and Theorem \ref{5.1}).

\smallskip
\section{Basic algebraic properties of power monoids of numerical monoids} \label{3}
\smallskip

In this section, we determine prime elements  and some further elementary properties of power monoids of numerical monoids (for absolutely irreducible elements, we refer to Theorem \ref{absIrr}, and for cancellative elements to Corollary \ref{important}). The first statement of Theorem \ref{3.1}.1 was observed at several places, but we repeat its simple proof  for convenience.

\smallskip
\begin{theorem} \label{3.1}~

\begin{enumerate}
\item $\{1\}$ is a cancellative prime element of $\mathcal P_{\fin} (\N_0)$, whence $\mathcal P_{\fin} (\N_0) = \big\{ \{k\} \colon k \in \N_0 \big\} \oplus \mathcal P_{\fin,0} (\N_0)$. No other element of $\Pfn$ is prime.

\item If $\mathcal P_{\fin, 0} (\N_0) = H_1 \oplus H_2$ for submonoids $H_1$ and $H_2$, then $H_1 = \big\{\{0\} \big\}$ or $H_2 = \big\{\{0\} \big\}$.
\end{enumerate}
\end{theorem}

\begin{proof}
1. To show that $\{1\}$ is a prime element, let $A, B \subset \N_0$ be finite nonempty such that $\{1\}$ divides $A + B$. Then either $\min (A) \ge 1$ or $\min (B) \ge 1$, whence $\{1\} \t A$ or $\{1\} \t B$. If $\{1\}+A = \{1\}+B$, then $A=B$, whence $\{1\}$ is cancellative. Thus, the structural statement on $\mathcal P_{\fin} (\N_0)$ follows from equation \eqref{structure}.

Now let $A \in \mathcal P_{\fin} (\N_0)$ be an irreducible element distinct from $\{1\}$, whence $A \in \mathcal P_{\fin, 0} (\N_0)$.
 Assume for a contradiction that $A$ is a prime element. Since
\[
A + [0, \max (A)] = [0, 2 \max (A)] = [0, \max (A)] + [0, \max (A)] \,,
\]
it follows that $A \t [0, \max (A)]$. This implies that $A = [0, \max (A)]$, whence $A= \{0, 1\}$ because $A$ is irreducible. Since $[0,6] = \{0,2,3\} + \{0,1,3\}$, $\{0,1\} \t [0,6]$, $\{0,1\} \nmid \{0,2,3\}$, and $\{0,1\} \nmid \{0,1,3\}$, we see that $A=\{0,1\}$ is not prime, a contradiction.

2. Suppose that $\mathcal P_{\fin, 0} (\N_0) = H_1 \oplus H_2$, where $H_1, H_2$ are submonoids.
It is easy to see that these are divisor-closed.
Since $\{0,1\}$ is an atom, we may  suppose without loss of generality that $\{0,1\} \in H_1$. Then $[0,k] = k \{0,1\} \in H_1$ for every $k \in \N$. Let $A \in  \mathcal P_{\fin, 0} (\N_0)$ with $A \ne \{0\}$. Since $A + [0, \max (A)] = [0, 2 \max (A)] \in H_1$ and $H_1 \subset \mathcal P_{\fin, 0} (\N_0)$ is divisor-closed, it follows that $A \in H_1$, whence $H_2 = \big\{ \{0\} \big\}$.
\end{proof}

The next result shows that $\Pf(S)$, for a numerical monoid $S\neq\N_0$, is quite different from $\Pf(\N_0)$ as it has no
prime elements and cannot be decomposed as a direct sum.

\smallskip
\begin{theorem} \label{3.2}
Let $S$ be a numerical monoid with $S \ne \N_0$.
\begin{enumerate}
\item The monoid $\mathcal P_{\fin} (S)$ has no prime elements.
\item If $\mathcal P_{\fin} (S) = H_1 \oplus H_2$ for some submonoids $H_1$ and $H_2$, then $H_1 = \{0\}$ or $H_2 = \{0\}$.

\item If $S'$ is a numerical monoid such that $\mathcal P_{\fin} (S')$ is isomorphic to $\mathcal P_{\fin} (S)$, then $S = S'$.
\end{enumerate}
\end{theorem}

\begin{proof}
1. First, we show that no singleton $\{k\}$ for $k\in S$ is prime.
This follows from the facts that the singletons form a divisor-closed submonoid isomorphic to $S$ and that the numerical monoid $S\neq\N_0$ has no prime element.

Now let $A\in \Pf(S)$ have cardinality at least $2$.
Let $a_1<a_2$ be two elements of $A$ and $b=a_2-a_1>0$.
If $m>0$ is large enough, the intervals $I=[2m-\min(A),4m-\max(A)]$ and $[m,2m]$ are included in
$S$ and $A+I=[2m,4m]$. Moreover,  $[2m,4m]= \big([m,2m]\setminus \{m+b\} \big)+ \big([m,2m]\setminus\{2m-b\} \big)$.
But $A$ does not divide $B=[m,2m]\setminus \{m+b\}$, since $m\in B$ but neither
$m+b$ nor $m-b$ are in $A$.
Similarly, $A$ does not divide $C=[m,2m]\setminus\{2m-b\}$, but it divides $B+C$,
whence  $A$ is not prime.

2. Assume $\mathcal P_{\fin} (S) = H_1 \oplus H_2$ for some submonoids $H_1$ and $H_2$.
Then these submonoids are necessarily divisor-closed.
Consider $a=\min S\cap (S-1)$.
Then $a>0$ because $1\notin S$. The set $A=\{a,a+1\}$ is by construction an atom in $\Pf(S)$. Therefore, $A\in H_1\cup H_2$, say $A\in H_1$.
Since $H_1$ is divisor-closed, this implies that $\LK A\RK\subset H_1$.
On the other hand, $kA=[ka,ka+k]$ whence,  for any $B\in\Pf(S)$ and any $k>\max (\max B,(\mathsf{F}(S)+\min(B))/a)$, we have $B+[ka-\min(B),ka-\max(B)]=kA$ and $[ka-\min(B),ka-\max(B)]\subset S$. Thus, it follows that
$B$ divides $kA$ in $\Pf(S)$, whence $B\in \LK A\RK$. Therefore, we infer that $\LK A\RK=\Pf(S)$, whence $H_1=\Pf(S)$ and $H_2= \big\{\{0\} \big\}$.

3. Let $S'$ be a numerical monoid and let $\varphi \colon \mathcal P_{\fin} (S) \to \mathcal P_{\fin} (S')$ be a monoid isomorphism. Since $\mathcal P_{\fin} (\N_0)$ has a prime element and $\mathcal P_{\fin} (S)$ has no prime elements, it follows that $S' \ne \N_0$. Clearly,
\[
H_0 = \big\{ \{k\} \colon k \in S \big\} \quad \text{and} \quad H_0' = \big\{ \{k\} \colon k \in S' \big\}
\]
are divisor-closed submonoids of $\mathcal P_{\fin} (S)$ and $\mathcal P_{\fin} (S')$. Furthermore, $H_0 \cong S$, $H_0' \cong S'$, and $H_0, H_0'$ are cancellative monoids. We continue with the following assertion.

\begin{enumerate}
\item[{\bf A1.}] $\varphi (H_0) = H_0'$.
\end{enumerate}

Suppose that {\bf A1} holds. Since isomorphic numerical monoids are equal, it follows that $H_0 = H_0'$, whence $S = S'$.

{\it Proof of {\bf A1.}}   Every monoid isomorphism $\psi \colon D \to D'$ maps divisor-closed cancellative submonoids of $D$ onto divisor-closed cancellative submonoids of $D'$. Thus, it remains to prove that $H_0$ is the only cancellative divisor-closed submonoid of $\mathcal P_{\fin} (S)$ distinct from $\big\{ \{0\} \big\}$. Let $H_1 \subset \mathcal P_{\fin} (S)$ be a divisor-closed submonoid, which is distinct from $\{ \{0\}\}$ and distinct from $H_0$. Since $H_0 \cong S$ and $S$ is primary, $H_0$ has no proper divisor-closed submonoids. Thus, $H_1$ is not contained in $H_0$, whence there is a subset $A \subset H_1$ with $|A| \ge 2$. But then $\LK A \RK \subset H_1$ is not cancellative,
so $H_1$ is not cancellative.
\end{proof}

Although the monoid $\Pf(S)$ characterises the numerical monoid $S$, its Grothendieck group does not depend on $S$.

\smallskip
\begin{theorem}
Let $S$ be a numerical monoid. Then the Grothendieck group of $\mathcal P_{\fin} (S)$ is isomorphic to $\Z\oplus\Z$ and the Grothendieck group of $\Pfze(S)$ is isomorphic to $\Z$.
\end{theorem}

\begin{proof}
Let $H$ be an additively written monoid. We define an equivalence relation on $H \times H$. We say that two pairs $(A, B)$ and $(C, D)$ are equivalent if there is $E \in H$ such that
\begin{equation}
\label{abe}
A+D+E = C+B+E \,.
\end{equation}
Let
\[
G = \{ [(A, B)] \colon A, B \in H \}
\]
be the set of equivalence classes of pairs $(A, B) \in H \times H$, and let
addition be defined via representatives. Then $G$ is an abelian group which satisfies the universal property characterizing the  Grothendieck group of $H$.

(i) Let  $H = \mathcal P_{\fin} (S)$.
Let $(A, B)$ and $(C, D)$ be pairs of elements of $H$. If there is $E \in H$ such that $A+D+E = C+B+E$, then $\max (A+D) = \max (C+B)$ and $\min(A+D)=\min(C+B)$.
Conversely, suppose $\max (A+D) = \max (C+B)=M$, and  $\min(A+D)=\min(C+B)=m$.  Then, for $E = [\mathsf F(S)+1, M-m+ \mathsf F (S)+1]\in H$, we have
\[
A+D + E = [\mathsf F(S)+1+m,\mathsf F(S)+1+2M-m] = C+B+E \,,
\]
whence $(A,B)$ and $(C,D)$ are equivalent. Since $\max$ and $\min$ are both additive functions, the pairs $(A,B)$ and $(C,D)$ are equivalent if and only if
\begin{align*}
\max (A) - \max (B) &= \max (C) - \max (D) \,\text{ and}\\
\min (A) - \min (B) &= \min (C) - \min (D) \,.
\end{align*}
Thus, the map $\phi : G\rightarrow\Z^2$ defined by $\phi([A,B])=(\max A-\max B,\min A-\min B)$ is well defined and a group homomorphism.
The arguments above show that it is injective, and it is surjective since
$\phi([\{\mathsf F(S)+1, \mathsf F(S)+2\},\{\mathsf F(S)+1\}])=(1,0)$ and similarly $(0,1)$ is attained.

(ii) Let $H = \mathcal P_{\fin} (S)$.
The proof runs along the same lines.
If $(A,B,C,D)\in H^4$ satisfy \eqref{abe}, then $\max (A+D) = \max (C+B)$.
Conversely, suppose $\max (A+D) = \max (C+B)=M$.
Then, for $E = S\cap[0, M+ \mathsf F (S)+1]\in H$, we have
$A+D + E =  C+B+E$,
whence $(A,B)$ and $(C,D)$ are equivalent.
Thus, the  map $G\rightarrow \Z$, defined by $(A,B)\mapsto \max A-\max B$, is an isomorphism.
\end{proof}

In the final result of this section, we observe that  studying divisor-closed submonoids of $\Pfn$ is equivalent to studying
divisor-closed submonoids of $\Pfz$, which we will do in Section \ref{4}. Furthermore, note that $\mathcal P_{\fin} (S)$, with $S \subsetneq \N_0$ being a numerical monoid, is a submonoid but not a divisor-closed submonoid of $\mathcal P_{\fin} (\N_0)$. Indeed, if $k\in S$,
then $\{k\}=k\{1\}$ yet $\{1\}\notin\Pf(S)$ and $\{k\}\in\Pf(S)$.

\smallskip
\begin{proposition} \label{3.3}
Let $H$ be a divisor-closed submonoid of $\Pfn$.
Then there exists  a divisor-closed submonoid $H'$ of $\Pfz$ such that $H=H'$ or $H=\{\{k\} : k\in\N_0\}\oplus H'$.
\end{proposition}

\begin{proof}
The monoid $H'=\{B-\min(B) : B\in H\}$ is  a divisor-closed submonoid of $\Pfz$.
If  $H\subset \Pfz$, then $H=H'$. Suppose that this is not the case. Then there exists $D\in H$ such that $0\notin D$. But, then $\{1\}$ divides $D$, whence $\{1\}\in H$ and $\{\{k\} : k\in\N_0\}\subset H$.
The decomposition $B=\{\min B\}+ (B-\min(B))$ for $B\in H$ then yields the result.
\end{proof}

\smallskip
\section{Prime spectrum and divisor-closed submonoids of restricted power monoids} \label{4}
\smallskip

In this section, we study the prime spectrum of restricted power monoids of numerical monoids. Let $S$ be a numerical monoid. A subset $\mathfrak p \subset \mathcal P_{\fin, 0} (S)$ is a prime $s$-ideal if and only if its complement $\mathcal P_{\fin, 0} (S) \setminus \mathfrak p$ is a divisor-closed submonoid. Furthermore, maximal divisor-closed submonoids correspond to minimal nonempty prime $s$-ideals.  Thus, we may   formulate our results on  the prime spectrum of power monoids in terms of divisor-closed submonoids.

Let $S_1 \subset S_2 \subset \N_0$ be  submonoids. Then $\mathcal P_{\fin, 0} (S_1) \subset \mathcal P_{\fin, 0} (S_2)$ is a divisor-closed submonoid, whence the divisor-closed submonoids of $\mathcal P_{\fin, 0} (S_1)$ are those divisor-closed submonoids of $\mathcal P_{\fin, 0} (S_2)$ that are contained in $\mathcal P_{\fin, 0} (S_1)$ (for overmonoids of a given numerical monoid see \cite{MR1978528, MR2564064}). Thus, when characterizing divisor-closed submonoids of restricted power monoids of numerical monoids, we may restrict to the restricted power monoid of the non-negative integers.
In particular, all monoids $\mathcal P_{\fin, 0} (S)$, where $S \subset \N_0$ is a submonoid, are divisor-closed submonoids of $\mathcal P_{\fin, 0} (\N_0)$.
But, they are not the only examples. Consider the set $A=\{0,1,3\}$. Since $\langle A \rangle = \N_0$, $\N_0$ is the only submonoid $S \subset \N_0$ with $\LK A\RK\subset \Pfze(S)$. However, since no set $C\in\Pfz$ with
 $\max C-1\in C$ divides a multiple of $A$, it follows that $\LK A \RK\neq \Pfz$.

 This example motivates the introduction of the \textit{reversion} operator as follows.
 For $B\in\Pfz$, we set $\rev (B):=\max B-B\in \Pfz$.
 This operator has the following  properties for all $B,C\in\Pfz$:
 \begin{itemize}
 \item $\rev(B+C)=\rev(B)+\rev(C)$;
 \item $B\mid C$ if and only if $\rev(B)\mid\rev(C)$;
 \item $\rev(\rev(B))=B$ (i.e., the reversion operator $\rev$ is involutive).
 \end{itemize}
For a subset $H\subset \Pfz$, we set $\rev(H):=\{\rev(B):B\in H\}$.
The  properties above show that $H$ is a submonoid, resp. a divisor-closed submonoid, of $\Pfz$,
if and only if $\rev(H)$ is a submonoid, resp. a divisor-closed submonoid, of $\Pfz$.
Thus, for any submonoid $S\subset\N_0$, the set $\rev(\Pfze(S))$ is a divisor-closed
submonoid of $\Pfz$ and, for any two submonoids $S,T\subset\N_0$,
 $\Pfze(S)\cap\rev(\Pfze(T))$
is a divisor-closed submonoid of $\Pfz$.

In the above example, we observe that $\rev(A)=\{0,2,3\}$, $\rev (A)$
generates the numerical monoid $\N_0\setminus\{1\}$, and
$\LK A\RK =\rev(\Pfze(\N_0\setminus\{1\}))$.
We will actually show that all divisor-closed submonoids of $\Pfz$ are of the form $\Pfze(S)\cap\rev(\Pfze(T))$.

Since we  study submonoids of the form $\LK A\RK$,
we  need to understand iterated sumsets.
Obviously, for any $n\in\N$, we have
\[
nA\subset \langle A\rangle \quad \text{and} \quad nA=\rev(n\rev(A))\subset n\max A-\langle\rev(A)\rangle \,.
\]
It turns out that these two inclusions characterise $nA$, for $n$ large enough in terms of $A$.

\smallskip
\begin{lemma} \label{4.1}
Let $A \in \mathcal P_{\fin, 0} ( \N_0)$ with $A \ne \{0\}$. Then there exists a constant $n^*(A) \in  \N$ such that, for all $n \ge n^* (A)$,
\[
nA = \langle A\rangle \cap \big(n\max A-\langle\rev(A)\rangle \big).
\]
\end{lemma}

\begin{proof}
See \cite[Theorem 1.1]{Na96b}.
\end{proof}

We would like to record that various aspects of the $n$-fold sumset have received wide attention in the literature. For recent progress, we refer to \cite{El22a, Co-El-MR22}. A tight bound for $n^* (A)$ was obtained in \cite{Gr-Wa21a}. We continue a simple consequence of the previous lemma.

\smallskip
\begin{corollary} \label{important}
The set of  cancellative elements of $\Pfn$ is $\big\{ \{k\} \colon k \in \N_0 \big\}$.
In fact, for any $A\in\Pf(\N_0)$ of cardinality at least 2,
$A$ is not even cancellative in $\LK A \RK$.
\end{corollary}

\begin{proof}
Clearly, for every $k \in \N_0$, the set $\{k\}$ is a cancellative element of $\Pfn$. Let $A \in \Pfn$ with $|A| \ge 2$. It suffices to show that $A$ is not cancellative  in $\LK A \RK$.
To do so, we assert that, for some $n\in\N$,  $(n+1)A=A+nA=A+B$ for some
$B\subsetneq nA$.
After shifting $A$ if necessary,  we may suppose that $0\in A$. Furthermore,   we may also suppose that $\gcd(A)=1$ and we set $a = \max A$.
Now let $n\ge n^* (A)$ be large enough so that  Lemma \ref{4.1} holds.
Then, there exist  finite sets $F,G\in\Pfz$ and natural integers $r,q$ such that
\[
nA=F\cup (r,an-q)\cup an-G \quad \text{ and} \quad  (n+1)A=F\cup (r, a(n+1)-q)\cup a(n+1)-G \,.
\]
We claim that $(n+1)A=nA\setminus\{a\}+\{0,a\}$.
Indeed, $(n+1)A\setminus\{a,2a\}\subset nA\setminus\{a\}+\{0,a\}\subset (n+1)A$
by definition, and it is easy to check that $a$ and $2a$ are in $nA\setminus\{a\}+\{0,a\}$ too (since $a\neq 0$ and $n\ge 2$).
Therefore, we obtain that $(n+1)A=A + (nA\setminus\{a\})$ as desired.
\end{proof}

We now provide a characterization of  divisor-closed submonoids of $\Pfz$.

\smallskip
\begin{theorem} \label{4.2}~
Let $H \subset \mathcal P_{\fin, 0} (\N_0)$ be a divisor-closed submonoid.
Then $H=\mathcal P_{\fin,0} (S)\cap \rev(\mathcal P_{\fin,0} (T))$, where $S=\langle B:B\in H\rangle$ and $T=\langle\rev(B) : B\in H\rangle$.
Moreover,  there exists  an $A\in H$ with $S=\langle A\rangle$ and $T=\langle \rev(A)\rangle$ and for any such $A$ we have $H=\LK A\RK$.
\end{theorem}

We first state and prove a technical auxiliary lemma which will turn useful again later.

\smallskip
\begin{lemma}\label{na}
Let $A\in\Pfz$ satisfy $\gcd(A)=1$. Let  $S=\langle A\rangle$, and $T=\langle \rev(A)\rangle$.
Then for all $B\in\Pfze(S)\cap\rev(\Pfze(T))$ and all $N\geq n^*(A)$ with
$\max\Delta(B)< N\max A-\mathsf{F} (S) -\mathsf{F} (T) -\max B$, we have $B \t NA$.
\end{lemma}

\begin{proof}
Let $B\in\Pfze(S)\cap\rev(\Pfze(T))$, $b=\max B$, and  $N\geq n^*(A)$ with
$\max\Delta(B)< N\max A-\mathsf{F}(S)-\mathsf{F}(T)-\max B$.
Setting $F=[0,\mathsf F (S)]\cap S$, $G=[0,\mathsf F (T)]\cap T$, and $a = \max A$, we observe that, by Lemma \ref{4.1} and the assumption $\gcd(A)=1$,
\begin{equation*}
NA=F\cup (\mathsf F (S),aN-\mathsf F (T))\cup aN-G
\end{equation*}
and we consider the set
\begin{equation*}
C=F\cup (\mathsf F (S),aN-\mathsf F (T)-b)\cup aN-b-G \,.
\end{equation*}
Then $B+F\subset \langle A\rangle \cap [0,\mathsf F (S)+b]\subset NA$ because $aN-\mathsf F (T)>\mathsf F (S)+b$, and thus
 $F\subset B+F\subset NA$.
Furthermore, we have $B+(\mathsf F (S),aN-\mathsf F (T)-b)=(\mathsf F (S),aN-\mathsf F (T))$ because $aN-\mathsf F (T)-b> \max \Delta(B)$.
Finally, we obtain that $aN-G\subset B+aN-b-G$ and
\[
b-B+G=\rev(B)+G\subset \langle \rev (A)\rangle
\cap [0,b+\mathsf F (T)]\subset N\rev(A)=aN-NA \,,
\]
because $N$ is so large  that $B+aN-b-G\subset NA$. It follows that $B+C=NA$.
\end{proof}

\smallskip
\begin{proof}[Proof of Theorem \ref{4.2}]
If $H = \{ \{0\} \}$, then all claims hold true.  Suppose that  $H \ne \{ \{0\} \}$ and set $d=\gcd \bigcup_{B\in H}B$.
Then $H=\{d\cdot B:B \in H'\}$ where $H' \subset \mathcal P_{\fin, 0} (\N_0)$ is a divisor-closed submonoid with
$1=\gcd \bigcup_{B\in H'}B$.
Furthermore,  for any $A\in\Pfz$, we have $\{d\cdot B : B\in \LK A\RK\}=\LK d\cdot A\RK$
and for any submonoids $S,T$ of $\N_0$ we have
$$\{d\cdot B : B\in  \mathcal{P}_{\fin,0} (S)\cap \rev(\mathcal P_{\fin,0} (T))\}=\mathcal P_{\fin,0} ( d\cdot S)\cap \rev(\mathcal P_{\fin,0} (d\cdot T)).$$
Therefore, we may  suppose without restriction that $d=1$.
Then $S = \langle B : B\in H\rangle$ is a numerical monoid and so is $T = \langle \rev(B) : B\in H\rangle$, because $\gcd(B)=\gcd(\rev(B))$ for any $B\in\Pfz$.

We assert that there is $A\in H$ such that $\langle A\rangle=S$
and $\langle \rev(A)\rangle=T$.
To find $B\in H$ with $\langle B\rangle=S$,
we construct a sequence $B_0,B_1, \ldots$ of elements of $H$ as follows:
let $B_0=\{0\}$ and
assuming $B_0,\ldots B_i$ are constructed and satisfy $\langle B_0,\ldots,B_i\rangle\neq S$, we take $B_{i+1}\in H$ such that $B_{i+1}\not\subset \langle B_0,\ldots,B_i\rangle$.
This process terminates, since there are no infinite increasing sequences
of submonoids of $\N_0$.
Let $n \in \N$ satisfies $\langle B_0,\ldots,B_n\rangle = S$ and
$B = \langle B_0+\cdots+B_n$,
 then we have $\langle B\rangle =\langle B_0,\ldots,B_n\rangle =S$.
 By a similar
iterative construction as above, there exists $C\in H$
such that $\langle\rev(C)\rangle =T$.
Setting  $A=B+C \in H$ we infer that $\langle A\rangle=S$ and $\langle\rev(A)\rangle=T$.

Now we fix some $A\in H$ for which  $\langle A\rangle=S$ and $\langle\rev(A)\rangle=T$, in particular $\gcd(A)=1$. We have
$$
\LK A \RK\subset H \subset \mathcal P_{\fin,0} (S)\cap \rev(\mathcal P_{\fin,0} (T)) \,,
$$
whence it remains to prove that $\mathcal P_{\fin,0} (S)\cap \rev(\mathcal P_{\fin,0} (T))\subset \LK A\RK$.

Let $B\in \mathcal P_{\fin,0} (S)\cap \rev(\mathcal P_{\fin,0} (T))$. Because of Lemma \ref{na}, we know that $B\mid NA$ for $N$ large enough, whence $B\in \LK A\RK$.
\end{proof}

\smallskip
\begin{theorem}~ \label{4.3}
\begin{enumerate}
\item For any $A\in\Pfz$, we have $\LK A\RK=\Pfz$  if and only if $1\in A\cap\rev(A)$.

\item The only maximal divisor-closed submonoids of $\Pfz$ are $\Pfze(\N_0\setminus\{1\})$ and $\rev(\Pfze(\N_0\setminus\{1\}))$. More generally, given numerical monoids $S_1$ and $S_2$, the maximal divisor-closed submonoids of $\Pfze(S_1)\cap \rev(\Pfze(S_2))$ are precisely the monoids $\Pfze(S'_1)\cap \rev(\Pfze(S_2))$ and $\Pfze(S_1)\cap \rev(\Pfze(S'_2))$ where $S'_1,S'_2$ are maximal submonoids of $S_1,S_2$ respectively.

\item There is a descending chain of divisor-closed submonoids of $   \mathcal P_{\fin, 0} (\N_0)$ that does not become stationary.

\item Every ascending chain of divisor-closed submonoids of $   \mathcal P_{\fin, 0} (\N_0)$ becomes stationary.
\end{enumerate}
\end{theorem}

\begin{proof}
1. In view of Theorem \ref{4.2}, we have $\LK A\RK=\Pfz$ if and only if $\langle A\rangle=\langle \rev(A)\rangle=\N_0$.

2. It suffices to prove the more general statement.
Let $S_1$ and $S_2$ be numerical monoids and let $H=\Pfze(S_1)\cap \rev(\Pfze(S_2))$.

First, we assert that $H'=\Pfze(S'_1)\cap \rev(\Pfze(S_2))$ is a maximal divisor-closed submonoid of $H$. Let $B\in H'$ such that $H'=\LK B\RK$.
Since $S'_1$ is a maximal submonoid of $S_1$, we infer that $S_1\setminus S'_1=\{a\}$, where $a$ is an atom of $S_1$.
If $A\in H\setminus H'$, then $a\in A$ and $\LK \{A\}\cup H'\RK$ contains $\LK A+B\RK$.
Since $\langle A+B\rangle=S_1$, Theorem \ref{4.2} implies that $\LK A+B\RK=H$.

Now we consider $H''=\Pfze(S_1)\cap \rev(\Pfze(S'_2))$. The case of $H'$ shows that
$\rev(H'')$ is a maximal divisor-closed submonoid of $\rev(H)$, whence $H''$ is a
maximal divisor-closed submonoid of $H$.

Let $H^* \subset H$ be any proper divisor-closed submonoid of $H$.
Since $H^*$ is a divisor-closed submonoid of $\Pfz$, Theorem \ref{4.2} implies that $H^*=\Pfze(S'_1)\cap\rev(\Pfze(S'_2))$ for some submonoids $S'1,S'_2\subset \N_0$. Since $H^*\subset H$, we infer that $S'_1,S'_2$ are submonoids of $S_1,S_2$ respectively.
If  $S'_i\neq S_i$ for both $i \in [1,2]$, then
$H^* \subset H$ is not maximal because $H^*\subsetneq \Pfze(S'_1)\cap \rev(\Pfze(S_2))\subsetneq H$.
But, $S'_i\neq S_i$ for some $i \in [1,2$], say $S'_1\neq S_1$ and $S'_2=S_2$.
Then there exists an atom $a\in S_1$ such that $S'_1\subset S_1\setminus\{a\}$.
If $S'_1\subsetneq S_1\setminus\{a\}$, then $H^*\subsetneq \Pfze(S_1\setminus\{a\})\cap \rev(\Pfze(S_2))\subsetneq H$, whence $H^*$ is not maximal.
Thus $S'_1 = S_1\setminus\{a\}$ and the claim follows.

3. For every $n \in \N$, we consider the numerical monoid $H_n = \{0\} \cup \N_{\ge n}$. Then $(H_n)_{n \ge 1}$ is a descending chain of numerical monoids that does not become stationary, and $\big( \mathcal P_{\fin, 0} (H_n) \big)_{n \ge 1}$ is a descending chain of divisor-closed submonoids of $\mathcal P_{\fin, 0} ( \N_0)$ that does not become stationary.

4. Let $(H_i)_{i\in\N}$ be an ascending chain of
divisor-closed submonoids of $\Pfz$. Let $S_i=\langle B : B\subset H_i\rangle$ and $T_i=\langle \rev(B) : B\subset H_i\rangle$. Then the sequences $(S_i)_{i\in\N}$
and $(T_i)_{i\in\N}$ are ascending chain of submonoids of $\N_0$, whence they both become stationary. Since $H_i=\Pfze(S_i)\cap\rev(\Pfze(T_i))$
by  Theorem \ref{4.2}, we infer that the sequence $(H_i)_{i\in\N}$ becomes stationary.
\end{proof}

\smallskip
Next we study algebraic properties of divisor-closed submonoids of $\mathcal P_{\fin,0} (S)$, where $S$ is a numerical monoid. Clearly, if a monoid is torsion-free, cancellative, locally finitely generated,  or transfer Krull, then the same is true for its divisor-closed submonoids. However, there are domains $D$, which are, for example,  not transfer Krull,  but whose divisor-closed submonoids $\LK a \RK$, for all $a \in D$, are transfer Krull (the domain $\Int ( \Z)$ has this property; for more we refer to \cite{Fr16a, Re14a, Re17a}). Furthermore, there are transfer homomorphisms from non-cancellative monoids to cancellative monoids (a transfer homomorphism from a not necessarily cancellative monoid of modules to a Krull monoid can be found in \cite{Ba-Sm21a}).

\smallskip
\begin{theorem} \label{4.4}
Let $S$ be a numerical monoid and let $\big\{ \{0\} \big\} \ne H \subset \mathcal P_{\fin,0} (S)$ be a divisor-closed submonoid. Then $H$ is neither  torsion-free, nor locally finitely generated. Moreover, there is no transfer homomorphism $\varphi \colon H \to B$, where $B$ is a cancellative monoid. In particular, $H$ is not transfer Krull.
\end{theorem}

\begin{proof}
Since $\mathcal P_{\fin,0} (S) \subset \mathcal P_{\fin,0} (\N_0)$ is a divisor-closed submonoid, $H$ is a divisor-closed submonoid of $\mathcal P_{\fin,0} (\N_0)$.
By Theorem \ref{4.2}, there are submonoids $S,T$ of $\N_0$ with $d = \gcd(S)=\gcd(T)$
such that $H=\mathcal P_{\fin,0} (S)\cap \rev(\mathcal P_{\fin,0} (T))$.
There exists $n^* \in \N$ such that $d\cdot \N_{\geq n^*}\subset S\cap T$.
Therefore, the set $F_n=\{0\}\cup d\cdot [n,5n]\cup \{6dn\}$
belongs to $H$ for any $n\geq n^*$.
Similarly, the sets $B_n=\{0\}\cup d\cdot [n,3n]\cup \{4dn\}$,
$D_n=\{0,dn\}$
and $C_n=\{0,2dn\}$ belong to $H$.
We denote by $B,C,D,F$ the sets $B_{n^*},C_{n^*},D_{n^*},F_{n^*}$, respectively.

(i) In order to show that $H$ is not finitely generated, we need to verify that it contains infinitely many atoms.
It suffices to observe that $C_n$ is an atom (like every pair) and belongs to $H$ for every $n\geq n^*$.

(ii) Assume to the contrary that there is a transfer homomorphism $\varphi \colon \LK A \RK \to M$, where $M$ is an additively written cancellative monoid. Observe that
the sets defined above satisfy $F=B+C=B+D+D$.
Then $\varphi (B) + \varphi (C) = \varphi (B+C)= \varphi (B+D+D) = \varphi (B)+\varphi (D)+ \varphi (D)$. Since $M$ is cancellative, we obtain that $\varphi (C) = \varphi (D) + \varphi (D)$. Since $\varphi$ is a transfer homomorphism, $\varphi (C)$ is an atom of $M$ and $\varphi (D)$ is not invertible in $M$, a contradiction.

(iii) Since the set $A=B\setminus\{2dn^*\}$ belongs to $H$ and $A+A=B+B=\{0\}\cup d\cdot [n,7n]\cup \{8dn\}$, the monoid $H$ is not torsion-free.
\end{proof}

If $S$ and $S'$ are numerical monoids such that $\mathcal P_{\fin} (S)$ and $\mathcal P_{\fin}(S')$ are isomorphic, then $S=S'$ by Theorem \ref{3.2}. We continue with the following conjecture.

\smallskip
\begin{conjecture} \label{4.5}
Let $S$ and $S'$ be numerical monoids. If $\mathcal P_{\fin, 0} (S)$ and $\mathcal P_{\fin, 0}(S')$ are isomorphic, then $S=S'$.
\end{conjecture}

Let $S$ be a numerical monoid. Since $S$ is cancellative but divisor-closed submonoids of $\Pfze(S)$ are not cancellative,
$S$ is not isomorphic to any divisor-closed submonoid of $\Pfze(S)$, whence $S$ cannot be retrieved so easily from $\mathcal P_{\fin, 0} (S)$.
Our next result
supports Conjecture \ref{4.5}.
 For a numerical monoid $S$, let $m(S)$ be the smallest integer $m\ge \max \mathcal{A}(S)$ such that $m$ and $m-1$ lie in $S$.

 \smallskip
\begin{theorem} \label{4.8}
Let $S$ be a  numerical monoid.
\begin{enumerate}
\item If $S'$ is a numerical monoid with  $\Pfze(S')$ being isomorphic to $\Pfze(S)$, then $\abs{\mathcal{A}(S)}=\abs{\mathcal{A}(S')}$ and $m(S)=m(S')$.

\item There are only finitely many numerical monoids $S'$ with $\Pfze(S')$ being isomorphic to $\Pfze(S)$.
\end{enumerate}
\end{theorem}

We need a lemma before embarking on the proof of Theorem \ref{4.8}. For a numerical monoid $S$ and for $A \in \mathcal P_{\fin} (S)$, let $\tau_S (A)$ denote the number of divisors of $A$ in $\mathcal P_{\fin} (S)$. If $A \in \mathcal P_{\fin,0} (S)$, then $\tau_S (A)$ equals the number of divisors of $A$ in $\mathcal P_{\fin, 0} (S)$. If $S = \N_0$, then we set $\tau (A) = \tau_{\N_0} (A)$ and any divisor of $A$ in $\Pfz$ is actually a divisor of $A$ in
$P_{\fin,0} (\langle A\rangle)\cap \rev(\mathcal P_{\fin,0} (\langle\rev(A)\rangle))$.
In what follows $\log$ denotes the logarithm of base 2.

\smallskip
\begin{lemma} \label{4.7}
Let $S$ be a numerical monoid and let $\{0\}\neq  A\in\mathcal P_{\fin, 0} (S)$.
Then $\log \big(\tau_S(nA) \big)$ is asymptotic to $n\max(A)/ \gcd (A)$ as $n$ tends to infinity.
\end{lemma}

\begin{proof}
We first handle the case $S = \N_0$. Let $A=d\cdot A'$ with $d = \gcd (A)$. Since the set of divisors of $A$ is the set of elements
$d\cdot B$ where $B$ is a divisor of $A'$, and since $\max(A')=\max(A)/d$, we
may assume that $d=1$ and $A=A'$.
On the one hand, every divisor of $nA$ in $\Pfz$ is a subset of $nA$, whence $\tau(A)\leq 2^{n\max(A)}$.
Now let $N \ge n^* (A)$ be large enough so that we can apply Lemma \ref{4.1}.
Let $b=N\max{A}-2\log N$. Taking $N$ large enough, we have $b>r+q$, where
$r=\mathsf{F}(\langle A\rangle)$ and $q=\mathsf{F}(\langle \rev(A)\rangle)$.
Let $D$ be a subset of $(r, b-q)$ and $B=\{0\}\cup D\cup\{b\}$, so $B\subset \langle A\rangle $ and $\rev(B)\subset \langle \rev(A)\rangle $.
Assume that $\max\Delta(B)<aN-q-b-r$.
Then by Lemma \ref{na}, $B$ is a divisor of $NA$.
Now we can estimate the number of divisors of $NA$.
There are $2^{b-q-r-1}\gg 2^{N\max A}/N^2$ subsets of $(r, b-q)$.
There are at most $b2^{b-(aN-q-b-r)}\ll 2^{N\max A}/N^3$ subsets of $[0,b]$ for which $\max\Delta(B)>aN-q-b-r$.
Therefore, there are $\gg 2^{N\max A}/N^2$ divisors of $NA$ as desired.

Now, let $S$ be a numerical monoid. Of course $\tau_S(A)\leq \tau(A)$.
Let $A=d\cdot A'$ with $d = \gcd (A)$. On the other hand $\tau_S(A)\geq \tau_S(A')$ since $d\cdot S\subset S$. Thus,  we may assume that $d=1$.
If $\mathsf g (S) = |\N_0 \setminus S|$ denotes the genus of $S$, then
we see that in the proof above there are at least
$2^{b-q-r-1- \mathsf g (S)}\gg 2^{N\max A}/N^2$ choices for $D$ such that $B=\{0\}\cup D\cup\{b\}\subset S$.
So we still have the asymptotic $\log \big(\tau_S(nA) \big)\sim n\max(A)/d$.
\end{proof}

\begin{proof}[Proof of Theorem \ref{4.8}]
1. Let $S'$ be a numerical monoid such that $\Pfze(S)$ and $\Pfze(S')$ are isomorphic.

According to Theorem \ref{4.3}.2,
the maximal divisor-closed submonoids of $\Pfze(S)$ are the monoids
$\Pfze(T)$ for $T\subset S$ maximal and $\Pfze(S)\cap\rev(\Pfze(\N_0\setminus\{1\})$.
Since the maximal submonoids of $S$ are precisely the monoids $S \setminus\{a\}$ for $a\in\mathcal{A}(S)$, there are exactly $\abs{\mathcal{A}(S)}+1$ maximal
divisor-closed submonoids of $\Pfze(S)$. Since the number of maximal
divisor-closed submonoids is invariant
under monoid isomorphism, we have  $\abs{\mathcal{A}(S)}=\abs{\mathcal{A}(S')}$.

For the second assertion,
observe that $A\in \Pfze(S)$ satisfies
$\LK A\RK=\Pfze(S)$ if and only if $\mathcal{A}(S)\subset A$ and $1\in\rev(A)$.
Thus $\max(A)\ge m(S)$, equality being possible by taking
$A=\mathcal{A}(S)\cup\{m(S)-1,m(S)\}$.
Fix this $A$ and let $\phi \colon \Pfze(S)\rightarrow\Pfze(S')$ be an isomorphism.
Observe that $\LK\phi(A)\RK=\phi(\LK A\RK)=\Pfze(S')$, which implies
that $\max(\phi(A))\ge m(S')$.
Applying Lemma \ref{4.7}, and since $\tau(A)=\tau(\phi(A))$, we may infer
$m(S)=\max(A)=\max(\phi(A))\ge m(S')$.
By symmetry, we obtain that $m(S)= m(S')$.

2. There are only finitely many numerical monoids whose atoms are all smaller than a given constant. Thus, the claim follows from the first item.
\end{proof}

\begin{remark}
Theorem \ref{4.8} shows that  $\Pfz$ is not isomorphic to  $\Pfze(S')$ for any numerical monoid $S' \subsetneq \N_0$, and the similar statement is true for
$\Pfze(\N_0 \setminus\{1\})$.
However, Theorem \ref{4.8} does not allow us to distinguish between $\Pfze(\langle 2,5\rangle)$ and $\Pfze(\langle 4,5\rangle)$.
Nevertheless, we can show that they are not isomorphic.
Let us abbreviate the phrase ``maximal divisor-closed submonoid''
to MDCS (plural MDCSs).
We note that $\Pfze(\langle 2,5\rangle)$ has three MDCSs, namely $\Pfze(\langle 2,7\rangle)$, $\Pfze(\langle 4,5,6,7\rangle)$ and $\Pfze(\langle 2,5\rangle)\cap\rev(\Pfze(\N_0\setminus\{1\} ))$, which in turn have respectively 3,4, and 4 MDCSs. The monoid
$\Pfze(\langle 4,5\rangle)$ also has 3 MDCSs, which have respectively 5, 5 and 4 MDCSs. So $\Pfze(\langle 2,5\rangle)$ and $\Pfze(\langle 4,5\rangle)$ are not isomorphic.
It is conceivable that such an argument, with possibly arbitrarily many iterations of the search for MDCSs, might yield a solution to Conjecture
\ref{4.5}. This is related to the study of  numerical semigroup trees (\cite{MR2564064, MR3128711}).
\end{remark}

Lemma \ref{4.7} will also be the crucial tool when studying absolutely irreducible elements of  power monoids of numerical monoids.
Absolutely irreducible elements in rings of integers in algebraic number fields are classic objects of interest. Theorem \ref{4.2} shows that all divisor-closed submonoids of restricted power monoids of numerical monoids have the form $\LK A \RK$ for some finite nonempty subset $A$ of the numerical monoid, and Theorem \ref{4.5} yields that $\LK A \RK$ is not Krull.  Recent work shows that absolutely irreducible elements are rather abundant in monoids $H$ with the property that all divisor-closed submonoids of the form $\LK a \RK$, with $a \in H$,  are Krull (Krull monoids have this property, but also rings of integer valued polynomials; see \cite{An22a, Ri-Wi21a,Fr-Na20a,Fr-Na-Ri22a}). In contrast to this, our next result shows that, for any numerical monoid $S \subsetneq \N_0$,  $\mathcal P_{\fin} (S)$ and $\mathcal P_{\fin, 0} (S)$ have no absolutely irreducible elements at all.

\smallskip
\begin{theorem} \label{absIrr}
Let $S$ be a  numerical monoid.
The element $\{1\} \in \mathcal P_{\fin} ( \N_0)$ is absolutely irreducible. Apart from this case, there are no further  absolutely irreducible elements  neither in $\Pfze(S)$ nor in $\Pf(S)$.
\end{theorem}

\begin{proof}
Since $\{1\} \in \mathcal P_{\fin} (\N_0)$ is a cancellative prime element by Theorem \ref{3.1}, it is absolutely irreducible.

Let $A\in \Pfz$ be irreducible with $\gcd (A)=d$.
Then, by Lemma \ref{4.7},  $\log \tau_S (NA)$ is asymptotic to $N\max(A)/d$. This implies, for $N$ large enough, that $NA$ has divisors in $\Pfze(S)$ that are different from $kA$ with $k \in [0, N]$, whence $|\mathsf Z (NA)|>1$.

Now let $A \in \mathcal P_{\fin} (S)$ be irreducible. We distinguish two cases. Suppose that $|A|=1$. If $S=\N_0$, then $A = \{1\}$, and we are back to the case handled above. Thus, suppose that  $S \ne \N_0$. Then $A = \{k\}$ for some $k \in \mathcal A (S)$, and there is $k' \in \mathcal A (S) \setminus \{k\}$. Then $\{kk'\} = k' \{k\} = k \{k'\}$, whence $A$ is not absolutely irreducible.

Now suppose that $|A| \ge 2$ and set $a=\min A$. Then $A-a\in\Pfz$ (but not necessarily in $\Pfze(S)$).
Then $\log\tau_S(N(A-a))$ is asymptotic to $n \big( \max (A) - \min (A) \big)/\gcd(A-a)$ by Lemma \ref{4.7}.
Let $B\subset S$ be a divisor of $N(A-\min(A))$. Then $B+N\min (A)$ divides $NA$
and is again included in $S$.
This shows that $NA$ has  more divisors in $\Pf(S)$ than those implied
by the factorization $NA=A+\cdots+A$, whenever $N$ is large enough, whence $\abs{\mathsf{Z}(NA)}>1$.
\end{proof}

\smallskip
\section{On catenary degrees and  $\omega$-invariants} \label{5}
\smallskip

In this section, we study finer arithmetic invariants (namely, catenary degrees and $\omega$-invariants), which consider not only the lengths of factorizations but do consider factorizations in a more direct way. Nevertheless, their behavior also controls the structure of sets of lengths (see Remark \ref{5.3}).

To recall definitions, let $H$ be a multiplicatively written atomic monoid. For an element $a \in H$, let $\omega (H,a)$ denote the smallest $N \in \N_0 \cup \{\infty\}$ with the following property:
\begin{itemize}
\item[] For all $n \in \N$ and $a_1, \ldots, a_n \in H$, if $a$ divides $a_1 \cdot \ldots \cdot a_n$, then there is a subset $\Omega \subset [1,n]$ with $|\Omega| \le N$ such that $a$ divides $\prod_{\lambda \in \Omega}a_{\lambda}$.
\end{itemize}
Thus, $a$ is a  prime element if and only if $\omega (H,a)=1$. For every $a \in H$, we have $\sup \mathsf L (a) \le \omega (H, a)$. In particular, if $\omega (H, a) < \infty$ for all $a \in H$, then all sets of lengths $L \in \mathcal L (H)$ are finite. We set $\omega (H) = \sup \{\omega (H,a) \colon a \in \mathcal A (H) \}$.

To define a distance function on the set of factorizations $\mathsf Z (H)$, let $z, z' \in \mathsf Z (H)$ be given, say
\[
z = u_1 \cdot \ldots \cdot u_k v_1 \cdot \ldots \cdot v_{\ell} \quad \text{and} \quad z' = u_1 \cdot \ldots \cdot u_k w_1 \cdot \ldots \cdot w_{m} \,,
\]
where $k, \ell, m \in \N_0$ and all $u_r, v_s, w_t \in \mathcal A (H_{\red})$ are such that $v_s$ and $w_t$ are pairwise distinct for all $s \in [1, \ell]$ and all $t \in [1,m]$. Then $\mathsf d (z,z')= \max \{\ell, m\} \in \N_0$ is the distance between $z$ and  $z'$, whence $\mathsf d (z,z')=0$ if and only if  $z=z'$. Let $a \in H$ and $M \in \N_0$. A finite sequence $z_0, \ldots, z_k \in \mathsf Z (a)$ is called an $M$-chain of factorizations if $\mathsf d (z_{i-1}, z_i) \le M$ for all $i \in [1,k]$. The catenary degree $\mathsf c (a)$ of $a$ is the smallest $M \in \N_0 \cup \{\infty\}$ such that any two factorizations of $a$ can be concatenated by an $M$-chain. Clearly, we have $\mathsf c (a) \le \sup \mathsf L (a)$. The {\it catenary degree} $\mathsf c (H)$ of $H$ is defined as the supremum over all $\mathsf c (a)$, whence
\[
\mathsf c (H) = \sup \{\mathsf c (a) \colon a \in H \} \,.
\]
If $\Delta (H) \ne \emptyset$, then $1 + \sup \Delta (H) \le \mathsf c (H)$ and if, in addition, $H$ is cancellative, then
\[
2 + \sup \Delta (H) \le \mathsf c (H) \quad \text{and} \quad \mathsf c (H) \le \omega (H) \,.
\]
If $H$ is finitely generated, then $\omega (H) < \infty$ by \cite[Proposition 3.4]{F-G-K-T17}. If $n \in \N_{\ge 2}$ and $H_n \subset \mathcal P_{\fin} (\N_0)$ is the submonoid generated by $\{0,1\}$ and by $\{1\} \cup (2 \cdot [0,n])$, then $\omega (H_n) = 2n+1$ (\cite[Remarks 3.11]{F-G-K-T17}). The next result shows, in particular, then $\omega \big( \mathcal P_{\fin}(\N_0) \big) = \infty$.

\smallskip
\begin{theorem} \label{5.1}
Let $S$ be a numerical monoid,
$a \in S$, and $A = \{0,a\}$. Then $\omega \big( \mathcal P_{\fin}(S) , A \big) = \omega \big( \mathcal P_{\fin, 0}(S) , A \big) = \infty$.
\end{theorem}

\begin{proof}
Let $n \in \N$ and $m \in [1,n]$. We show that $\omega \big( \mathcal P_{\fin, 0}(S) , A \big) \ge n+2$. Since $\mathcal P_{\fin, 0} (S) \subset \mathcal P_{\fin} (S)$ is divisor-closed, it follows that $\omega \big( \mathcal P_{\fin}(S) , A \big) \ge n+2$, whence the assertion follows.

Clearly, the elements
\[
\{0, 2a\}, \ \{0, 2a, 3a\} \quad \text{and} \quad \{0, a, a(2n+5)\}
\]
are atoms of $\mathcal P_{\fin, 0}(S)$. We have
\[
\begin{aligned}
A_{m,n} & := m\{0,2a\} + \{0,2a, 3a\} + \{0,a,a(2n+5)\} \\
 & = \big( \{0\} \cup a \cdot [2, 2m+3] \big) + a \cdot \{0, 1, 2n+5\} \\
 & = a \cdot \big( [0, 2m+4] \cup \{2n+5\} \cup [2n+7, 2m+2n+8] \big) \,.
\end{aligned}
\]
Since
\[
\begin{aligned}
A_{n,n} & = a \cdot \Big( [0, 2n+5] \cup [2n+7, 4n+8]   \Big) \\
  & = A + a \cdot \Big( [0, 2n+4] \cup [2n+7, 4n+7]  \Big) \,,
\end{aligned}
\]
$A$ divides $A_{n,n}$ in $\mathcal P_{\fin} (S)$, which is the sum of $n+2$ atoms. We assert that $A$ does not divide any proper subsum. Clearly, every subsum that is divisible by $A$ contains the atom $\{0, a, a(2n+5)\}$. If $m < n$, then the above calculation shows that $A_{m,n}$ is not divisible by $A$.
\end{proof}

\smallskip
In contrast to the previous result, we observe that, for all $A \in \mathcal P_{\fin, 0}(S)$, the catenary degree $\mathsf c (A)$ is finite. Moreover, by \cite[Theorem 4.11]{Fa-Tr18a}, we have
\[
\{ \mathsf c (A) \colon A \in \mathcal P_{\fin, 0}(S) \} = \{ \mathsf c (A) \colon A \in \mathcal P_{\fin}(S) \} = \N \,.
\]
Similarly, if $D$ is a Dedekind domain with infinitely many maximal ideals of finite index, then
\[
\{ \mathsf c (f) \colon f \in \Int (D) \} = \N_{\ge 2} \,;
\]
this follows immediately from the main result in \cite{Fr-Na-Ri19a} but is not stated explicitly. The set of catenary degrees of Krull monoids highly depends not only of the class group but on the distribution of prime divisors in the classes \cite{Ge-Zh19a}.

Atomic monoids, for which all sets of lengths are finite,  satisfy the ascending chain condition on principal ideals  \cite[Theorem 2.28 and Corollary 2.29]{Fa-Tr18a}. Thus, power monoids of numerical monoids satisfy the ascending chain condition on principal ideals.
The following corollary  considers the ascending chain condition for more general ideals (for the concept of ideal systems we refer to \cite{HK98}). Recall that  Krull monoids (in particular,  Krull domains) satisfy the ascending chain condition on divisorial ideals.

\smallskip
\begin{corollary} \label{5.2}~

\begin{enumerate}
\item Let $S$ be a numerical monoid, $r$ be a weak ideal system on $\mathcal P_{\fin, 0}(S)$ resp. on $\mathcal P_{\fin}(S)$ such that all principal ideals are $r$-ideals. Then $\mathcal P_{\fin, 0}(S)$ resp.  $\mathcal P_{\fin}(S)$ do not satisfy the ascending chain condition on $r$-ideals.

\item Let $D$ be a Dedekind domain with infinitely many maximal ideals of finite index. Then $\omega \big( \Int (D), X \big)= \infty$. Moreover, if  $r$ is a weak ideal system on $\Int (D)$ such that all principal ideals are $r$-ideals, then $\Int (D)$ does not satisfy the ascending chain condition on $r$-ideals.
\end{enumerate}
\end{corollary}

\begin{proof}
1. If a monoid $H$ satisfies the ascending chain condition on $r$-ideals, then $\omega (H,a) < \infty$ for all $a \in H$ by \cite[Proposition 3.3]{F-G-K-T17}. Thus, the assertion follows from Theorem \ref{5.1}.

2. Let $K$ denote the quotient field of $D$, whence $D[X] \subset \Int (D) \subset K[X]$. In \cite[Theorem 2]{Fr-Na-Ri19a}, it is proved that, for every $n \in \N$, there are irreducible polynomials $h, g_1, \ldots, g_{n+1} \in \Int (D)$ such that
\[
X h = g_1 \cdot \ldots \cdot g_{n+1} \,.
\]
This implies that $\omega \big( \Int (D), X \big) \ge n$ for all $n \in \N$, whence  $\omega \big( \Int (D), X \big) = \infty$. Now, again by \cite[Proposition 3.3]{F-G-K-T17}, it follows that $\Int (D)$ does not satisfy the ascending chain condition on $r$-ideals.
\end{proof}

In the final remark we discuss the significance of the finiteness of catenary degrees and $\omega$-degrees for the structure of sets of lengths. Among others, these connections show that Theorem \ref{5.1} supports Conjecture \ref{1.1}.

\smallskip
\begin{remark} \label{5.3}~
All results on the structure of sets of lengths, achieved so far,  can be divided into two classes. To begin with, algebraic finiteness conditions on various classes on monoids imply that sets of lengths are highly structured (for a survey see \cite[Chapter 4]{Ge-HK06a}). The typical example of such a class of monoids are Krull monoids with finite class group. But, all classes of monoids, for which such results are established, have finite catenary degree and they are locally tame (a property which, by definition, implies that $\omega (H,a) < \infty$ for all $a \in H$). In particular, if  $H$ is cancellative and $\omega (H) < \infty$, then sets of lengths in $H$ are highly structured (\cite{Ge-Ka10a}).

The extremal case on the other side of the spectrum is that every finite subset $L \subset \N_{\ge 2}$ occurs as a set of lengths. This holds true for Krull monoids with infinite class group and prime divisors in all classes, for rings of integer-valued polynomials $\Int (D)$, with $D$ as in Corollary \ref{5.2}, some primary monoids, some weakly Krull domains, and others (see \cite{Ka99a, Fr13a, Fr-Na-Ri19a,  Ch-Fa-Wi22a}, \cite[Theorem 3.6]{Go19a},\cite[Theorem 4.4]{Fa-Zh22a}). All these monoids and domains have infinite catenary degree and infinite $\omega$-invariant.
\end{remark}

\smallskip
\section{On the density of atoms} \label{6}

Quantitative aspects of arithmetic properties of atomic monoids have been studied since the very beginning of factorization theory. To start with the oldest strand of investigations in this direction, let $H$ be an atomic monoid with a suitable norm function $\mathsf N \colon H \to \N$ that allows to develop a theory of $L$-functions (rings of integers in algebraic number fields with the usual norm are the classic example). For $k \in \N$, the counting functions
\[
\mathsf P_k (x) = \# \{ a \in H \colon \mathsf N (a) \le x, a \ \text{satisfies Property} \ \mathsf P_k \} \,.
\]
are studied, among others, for the following arithmetic properties $\mathsf P_k$:
\[
\max \mathsf L (a) \le k, \quad \text{or} \quad  |\mathsf Z (a)| \le k, \quad \text{or} \quad  |\mathsf L (a)| \le k \,.
\]
Note that $\max \mathsf L (a) =1$ if and only if $\mathsf L (a)=\{1\}$ if and only if $a$ is an atom
(see the presentations in the monographs \cite[Chapter 9]{Na04}, \cite[Chapters 8 and 9]{Ge-HK06a}, or \cite{Ka17a, Ra22a} for more recent work). In particular, the density of elements $a$ (say in the ring of integers of an algebraic number field), whose sets of lengths $\mathsf L (a)$ are intervals, is equal to one (\cite[Theorem 9.4.11]{Ge-HK06a}).
Apart from this analytic strand of investigations, there are counting results for the number of atoms dividing the powers $a^n$ for elements $a \in H$ (e.g., \cite{HK93j, Ka20a, B-G-B-S22}),  for the number of atoms in numerical monoid algebras (\cite{A-E-K-ON-T22}), and others.

In all results so far, the density of atoms (in the respective sense) was always equal to zero. For restricted power monoids of numerical monoids, the contrary holds: the density of non-atoms equals zero and the density of atoms equals one.

\smallskip
\begin{theorem} \label{th:density}
Let $S$ be a numerical monoid.

\begin{enumerate}
\item If  $H=\Pfze(S)$, then
      \begin{equation} \label{density1}
      \lim_{x \to \infty} \frac{\# \{ A \in \mathcal A (H) \colon \max (A) \le x \}}{\# \{ A \in H \colon \max (A) \le x \}} = 1 \,.
       \end{equation}

\item If $H=\Pf(S)$, then
      \begin{equation} \label{densityG}
      \lim_{x \to \infty} \frac{\# \{ A \in \mathcal A (H) \colon \max (A) \le x \}}{\# \{ A \in H \colon \max (A) \le x \}} \in [1/2,1)\cap\Q \,.
      \end{equation}
      This latter limit is 1/2 if and only if $S=\N_0$.
\end{enumerate}
\end{theorem}

The first item of  Theorem \ref{th:density} already follows from a result of Shitov
\cite{shitov}, stated in the language of boolean polynomials. Shitov answered
a question of Kim and Roush \cite{kimroush} and confirmed an equivalent conjecture \cite[Conjecture 10]{dismal} stated in the language of ``dismal arithmetic''.
We are thankful to S.~Tringali for pointing out these references to us.
However, our result  is quantitatively superior as the decay rate of the number of non-atoms is vastly improved (see Theorem  \ref{proba0} and Remark \ref{quantitative}), and our method
is different. In particular, we do not require the language of ``irreducible boolean polynomials'', but we think of Theorem \ref{th:density} in a probabilistic way. Let the probability be denoted by $\Pp$, and
let us think about random subsets $A$ of $[0,N]$ as follows: we have
$A=\{n\in [0,N]:\xi_n=1\}$,  where $(\xi_n)_{n\leq N}$ is a sequence of independent and identically distributed random variables satisfying $\Pp(\xi_n=0)=\Pp(\xi_n=1)=1/2$.
A sufficient reference for the very elementary probability theory we will use is \cite{lesigne}.
 We now state our result within this probabilistic framework.

\smallskip
\begin{theorem} \label{proba0} \label{6.1}
Almost no set $A\subset [0,N]$ is a genuine sumset.
More precisely, let $\Dec(N)$ be the set of subsets
$A\subset [0,N]$ such that there exist
$B,C\subset [0,N]$ satisfying $A=B+C$ and $\min(\abs{B},\abs{C})\geq 2.$
Then
\[
\Pp_{A\subset [0,N]}(A\in\Dec(N))=\exp(-\Omega(N)) \,.
\]
\end{theorem}

\smallskip
\begin{remark} \label{quantitative}
Thus, the number of subsets of $[0,N]$,  which are genuine sumsets, is bounded above by $c^N$ for some $c<2$.
The bound of Shitov \cite{shitov} is of the form $2^{n-o(n)}$, which is much larger than ours.
Considering that for $B=\{0,1\}$, the sets $B+C$ for $C\subset 2\cdot [0,\lfloor (N-1)/2\rfloor]$ are pairwise distinct and yield $\Omega(2^{N/2})$ sets which are not atoms, we conclude that $c\geq \sqrt{2}$.
In fact, it follows from \cite[Theorem 3.3]{shitov} that $c\geq 2^{0.811}\geq 1.754$.
It seems difficult to estimate the correct value of $c$.
\end{remark}
\smallskip
We will derive Theorem \ref{th:density} from Theorem \ref{proba0}.
The latter looks very similar to a theorem of Wirsing \cite{Wi53}, which states that almost all infinite subsets of $\N$
are not equal to a sumset $B+C$ with $\min (\abs{B},\abs{C}) \ge 2$.
However, it does not seem that  the  result on infinite sets implies Theorem \ref{6.1}, nor does Wirsing's proof  directly yield Theorem \ref{6.1}. Nevertheless, our proof is inspired from his. We refer to \cite{bienvenu} for further probabilistic results about
additive decomposability.

It will be convenient to work on the zero-one sequence $(\xi_n)_{n\leq N}$ itself.
We need a couple of probabilistic lemmas for such sequences, or more generally for sequences over a finite alphabet.

If $N \in \N$, $a=(a_1,\ldots,a_N)\in X^N$ is a sequence over a finite alphabet $X$, and $z\in X$, then
\[
A(z)=\abs{\{n\in [1,N] : a_{n}=z\}}
\]
denotes the number of occurrences of $z$ in $a$.
 The probability distribution we consider on $X^N$ is the uniform distribution, that is,
the random variables $a_i$ are independent and uniformly distributed on $X$.
We will often encounter the function $f(g,\varepsilon)=\frac{(1-g^{-1})^{1-\gamma}}{(1-g\varepsilon)^\gamma (1-\gamma)^{1-\gamma}}$
for integers $g\geq 2$ and $\varepsilon\in [0,1/g)$, where $\gamma=1/g-\varepsilon$.

\smallskip
\begin{lemma} \label{deviation}
Let $X$ be a finite alphabet of order $\abs{X} = g \geq 2$ and let $\varepsilon\in (0,1/g)$.
Then, for any $s\in X$,
$$\Pp(\abs{A(s)/N-1/g}>\varepsilon)\ll_{g,\varepsilon} N^{1/2}f(g,\varepsilon)^N.$$
\end{lemma}
We will only ever use this bound in the regime where $g$ and $\varepsilon$ are constant.
Then $f(g,\varepsilon)\in (0,1)$ is a constant, whence the right hand side decays exponentially as $N$ tends to infinity.

\begin{proof}
For any $\nu\leq N$ the number of sequences in $X^N$ which have exactly $\nu$ occurrences of the symbol $s\in X$ is $\binom{N}{\nu}(g-1)^{N-\nu}$.
Note that, when $\nu$ ranges from $0$ to $N$, this increases from $0$ to $\nu=\lfloor N/g\rfloor$ and decreases from $\lceil N/g\rceil$ to $N$. Let $\gamma=1/g-\varepsilon$.
Let us bound $\Pp(A(s)/N<\gamma )$, the deviation in the other direction being analogous.
We have $\Pp(A(s)/N<\gamma)\leq g^{-N}\sum_{\nu=0}^{\lfloor \gamma N\rfloor}\binom{N}{\nu}(g-1)^{N-\nu}$. By monotonicity of the summand,
$$\Pp(A(s)/N<\gamma)\leq Ng^{-N}\binom{N}{\lfloor \gamma N\rfloor}(g-1)^{N(1-\gamma)}=
N\binom{N}{\lfloor \gamma N\rfloor}g^{-\gamma N}(1-g^{-1})^{(1-\gamma)N}\, .$$
Applying Stirling's formula we conclude.
\end{proof}

Given $b_1<\ldots<b_r\leq N$ forming a sequence $b$,
and $(z_1,\ldots,z_r)\in X^r$, let $A(z;b)$ be the number of $n\leq N-b_r$ such that
$a_{n+b_i}=z_i$ for all $i\in [1,r]$.

\smallskip
\begin{lemma} \label{pattern}
Let $r, N \in \N$,  $\delta\in (0,1)$, $z\in \{0,1\}^r$, and $B=\{b_1, \ldots, b_r\} \subset \N_0$
with $0\leq b_1<\ldots<b_r\leq N$. Then
\[
\Pp(\abs{A(z;b)-(N-b_r)2^{-r}}>\delta N)\ll \exp(-cN) \,,
\]
for some positive constant $c$ depending on $\delta$ and $r$ only.
\end{lemma}

Note that the statement is trivially true if $b_r>(1-\delta)N$.

\begin{proof}
Let $G$ be the graph whose vertex set is $[0,N]$ and where any two elements $n,m$ are connected
if $(n+B)\cap (m+B)\neq\emptyset$, equivalently $n-m\in B-B$.
Obviously, the degree of each vertex is at most $\abs{B-B}\leq r^2$.
As a result, the chromatic number is at most $r^2$.
So we can split $[0,N-b_r]$ into $\chi=\chi(G)\leq r^2$ classes $C^{(1)},
\ldots,C^{(\chi)}$ such that for each $i\in [\chi]$, the translates $n+B$, $n\in C^{(i)}$ are disjoint.
Let $d^{(i)}=((a_{n+b_1},\ldots,a_{n+b_r}))_{n\in C^{(i)}}$.
We can see each $d^{(i)}$ as a sequence of length $\abs{C^{(i)}}$ of elements of the alphabet $X=\{0,1\}^r$, which consists of $g=2^r$ symbols.
Also for each $i\in [\chi]$, the random variables $(a_{n+b_1},\ldots,a_{n+b_r})$ for $n\in C^{(i)}$ are pairwise independent by construction.
As usual, we denote for every $z\in X$
the number of $n\in C^{(i)}$ satisfying $(a_{n+b_1},\ldots,a_{n+b_r})=z$ by $D^{(i)}(z)$.

Now
$A(z;b)=\sum_{i\leq\chi} D^{(i)}(z)$.
Let $\eta$ and $\varepsilon$ be real numbers in $(0,1)$ to be specified later.
Let $M=\eta N$.
Let $k$ the number of $i\in [1, \chi]$ such that  $\abs{C^{(i)}}\geq M$.
Upon reordering, we may assume without loss of generality that $\abs{C^{(i)}}\geq M$
if, and only if, $i\in [1, k]$.
Then $A(z;b)=\sum_{i\leq k} D^{(i)}(z)+O(\eta r^2 N)$
and  by Lemma \ref{deviation}, we have
\[
\Pp(\abs{D^{(i)}(z)/\abs{C^{(i)}}-g^{-1}}>\varepsilon)\ll M^{1/2}f(g,\gamma)^M
\]
for $i\in [1, k]$.
The right-hand side above is indeed smaller than a constant times
$\exp(-cN)$ for some positive constant $c$ depending on $\varepsilon,\eta$ and $r$ only.
Therefore with probability $1-O(\exp(-cN))$
we have $$\sum_{i\leq k} D^{(i)}(z)=\sum_{i\leq k} \abs{C^{(i)}}g^{-r}+O(r^2\varepsilon N)$$
in which case
$$A(z;b)
=(N-b_r)g^{-r}+O(r^2(\eta+\varepsilon) N).$$
Taking $\eta=\varepsilon=c'\delta/(2r^2)$ for some suitable constant $c'>0$, we obtain the assertion.
\end{proof}

\smallskip
\begin{proof}[Proof of Theorem \ref{proba0}]
Fix some large integer $N$.
We write $\Dec(N)=\Dec_1\cup\Dec_2\cup\Dec_3$ where $\Dec_1,\Dec_2,\Dec_3$ are the set of subsets $A\subset [0,N]$ of the form $A=B+C$ for some sets $B,C\subset\N$ satisfying $\min(\abs{B},\abs{C})\geq 2$ and, respectively,
\begin{enumerate}
\item[(i)]  $\abs{B}+\abs{C}<N/5$
\item[(ii)] $\min(\abs{B},\abs{C})\in \{2,3\}$
\item[(iii)]  $\abs{B}+\abs{C}\geq N/5$ and $\min(\abs{B},\abs{C})\ge 4$.
\end{enumerate}
For the first item, note that
\[
\Pp_{A\subset [0,N]}(A\in \Dec_1)\leq 2^{-N}\binom{2N}{N/5} \,,
\]
which by Stirling's Formula can be shown to be $O(0.96^N)$.

Let us tackle $\Dec_2$. Let $B\subset [0,N]$ have precisely two or three elements. Let $C\subset [0,N]$ and $A=B+C$.
If $\max(B)\geq N/3$, and if $A\subset [0,N]$,
then $C\subset [0,2N/3]$. So there are at most $N^32^ {2N/3}$
such sumsets $B+C$.
Now assume $\max(B)\leq N/3$. Let $B=\{0,b_1,b_1+b_2\}$ where $b_1>0,b_2\ge 0$ are integers.
Since $A=B+C$, for every $n\in A$, either $n-b_2\in A$ or $n+b_1\in A$ or $n-b_1\in A$. Thus considering
$f=1_A:[0,N]\rightarrow \{0,1\}$, one can see that
$(f(n-b_1),f(n),f(n+b_1),f(n-b_2))$ can not equal $(0,1,0,0)$.
But for a random sequence $f : [0,N]\rightarrow \{0,1\}$, the number of $n\in [\max(b_1,b_2),N-b_1]$ (interval of length at least $N/3$) such that
$(f(n-b_1),f(n),f(n+b_1),f(n-b_2))=(0,1,0,0)$
is 0 with probability $\exp(-\Omega(N))$ by Lemma \ref{pattern}.
This holding for any such $B$, for which there are at most $N^3$ choices, we conclude that
$
\Pp_{A\subset [0,N]}(A\in \Dec_2)\leq \exp(-\Omega(N)).
$

We move on to $\Dec_3$. Let $B,C\subset [0,N]$ satisfy
$\abs{B}+\abs{C}\geq N/5$ and $\min(\abs{B},\abs{C})\ge 4$.
Let $A=B+C$.
We may assume by symmetry and pigeonhole principle that $\abs{B}\geq N/10$, $\abs{C}\geq 4$.
Let $D\subset C$ have cardinality 4.
Then for any $n\in B$, we have $n+D\subset A$.
Therefore, the number of $n\leq N$ such that $n+D\in A$ is at least  $N/10$; this is much more than the expected $(N-\max D)/16\leq N/16$, so by Lemma \ref{pattern} $\Pp_{A\subset [0,N]}(A\in\Dec_3)\leq\exp(-\Omega(N)).$\end{proof}

\begin{proof}[Proof of Theorem \ref{th:density}]
Let $S$ be a numerical monoid.

1. Let $H=\Pfze(S)$.
 Observe that the number of sets $A\in H$ which are not
atoms and satisfy $\max(A)\leq x$ is at most the number of sets $A\subset [0,x]$ such that there exist $B,C\subset [0,N]$ both of cardinality at least two such that $A=B+C$,
whereas the number of sets $A\in H$ which satisfy $\max(A)\leq x$
is bounded below by a constant (depending on $H$) times
the number of sets $A\subset [0,x]$.
Applying Theorem \ref{proba0} yields equation \eqref{density1}.

2. Let $H=\Pf(S)$.
Observe that an element of $H$ which is not an atom is either
a sum set $B+C$ with $\min(\abs{B},\abs{C})\geq 2$ or
a set of the form $\{k\}+B$ for some $k\in S\setminus\{0\},B\in H\setminus\{\{0\}\}$.
The non-atoms of the first kind have density 0 because of Theorem \ref{6.1}.
The non-atoms of the second kind are precisely the elements of $H'=\bigcup_{a\in\mathcal{A}(S)}\Pf(a+S)\setminus\{\{a\}\}$.
If $S=\N_0$, we have $H'=\Pf(\N)\setminus\{\{1\}\}$,
so that the limit in equation \eqref{densityG} exists and is equal to 1/2.
Otherwise, let $d(S)=\max\mathcal{A}(S)+ \mathsf{F}(S)$ and observe that a set $A\in H$ satisfies $A\in H'$ if and only if
$A\cap [0, d(S)]\in H'\cup\{\emptyset\}$; indeed, $[d(S)+1,+\infty)\subset \bigcap_{a\in\mathcal{A}(S)} (a+S)$.
Therefore the limit in equation \eqref{densityG} exists and is equal
to $(1+\abs {H'\cap \mathcal P ([0, d(S)])})
2^{-\abs{S\cap  [0, d(S)]}}$.
This is less than 1/2 since $H'\subset \Pf(S\setminus\{0\})$ and $H'$ misses $\mathcal{A}(S)\subset S\cap [1,d(S)]$, and obviously positive.
\end{proof}

\medskip
\noindent
{\bf Acknowledgement.} We would like to thank the reviewers for their careful reading and all their comments.

\medskip
\noindent
{\bf Note added in proof.} When this article went to press, we were informed that Salvatore Tringali and Weihao Yan gave an affirmative answer to Conjecture \ref{4.5}.

\providecommand{\bysame}{\leavevmode\hbox to3em{\hrulefill}\thinspace}
\providecommand{\MR}{\relax\ifhmode\unskip\space\fi MR }
\providecommand{\MRhref}[2]{%
  \href{http://www.ams.org/mathscinet-getitem?mr=#1}{#2}
}
\providecommand{\href}[2]{#2}

\end{document}